\documentclass[12pt]{article}%
\usepackage{lineno}
\usepackage{rotating}
\usepackage{pdflscape}
\usepackage{threeparttable}
\usepackage{graphicx,epstopdf}
\usepackage{tabularx}
\usepackage{amsmath}
\usepackage{accents}
\usepackage{amsfonts}
\usepackage{amssymb}%
\usepackage{multirow}
\usepackage{array}
\usepackage{setspace}
\usepackage{scalefnt}
\usepackage{amsthm}
\usepackage[affil-it]{authblk}

\usepackage[urlcolor=blue, 
linkcolor=blue, 
colorlinks=true,
citecolor=blue]{hyperref}

\usepackage[utf8]{inputenc}
\usepackage{booktabs}
\providecommand{\U}[1]{\protect\rule{.1in}{.1in}}
\providecommand{\U}[1]{\protect\rule{.1in}{.1in}}
\providecommand{\U}[1]{\protect\rule{.1in}{.1in}}
\newtheorem{remark}{Remark}
\newtheorem{proposition}{Proposition}
\newtheorem{definition}{Definition}
\newtheorem{theorem}{Theorem}

\usepackage{t1enc}

\begin{document}
\title{Human behaviors: a threat for mosquito control?}
\author[$^{\dagger}$]{Yves Dumont\thanks{E-mail: \href{yves.dumont@cirad.fr}{yves.dumont@cirad.fr}}}
\author[$^{\star}$]{Josselin Thuilliez \thanks{E-mail:\href{josselin.thuilliez@univ-paris1.fr}{josselin.thuilliez@univ-paris1.fr}}}
\affil[$^{\dagger}$]{ CIRAD UMR AMAP. TA A51/PS2 34398 Montpellier cedex 5.} 
\affil[$^{\star}$]{CNRS - Paris 1 Panth\'{e}on Sorbonne University. Centre d'\'economie de la Sorbonne. 106-112, Boulevard de l'H\^{o}pital. 75013 Paris
}

\maketitle
\begin{abstract} \linespread{1}
Community involvement and the preventive behavior of households are considered to be at the heart of vector-control strategies. In this work, we consider a simple theoretical model that enables us to take into account human behaviors that may interfere with vector control. The model reflects the trade-off between perceived costs and observed efficacy. Our theoretical results emphasize that households may reduce their protective behavior in response to mechanical elimination techniques piloted by a public agent, leading to an increase of the total number of mosquitoes in the surrounding environment and generating a barrier for vector-borne diseases control. Our study is sufficiently generic to be applied to different arboviral diseases. It also shows that vector-control models and strategies have to take into account human behaviors.


\bigskip
\noindent \textit{\textbf{AMS}}: 34C12, 91BXX, 92C60
\end{abstract}

\renewcommand{\thefootnote}{\roman{footnote}} 
\newpage
\setcounter{page}{1}
\setcounter{footnote}{0} 
\renewcommand{\thefootnote}{\arabic{footnote}} 
\section{Introduction}
Many preventive technologies are among the most inexpensive ways to promote good health and are often cheaper than curative healthcare. However, how people make decisions about the use of low-cost preventive technologies remains unclear. Education may influence the adoption of health technologies or counteract the natural tendency to do the wrong thing \cite{kremer_illusion_2007, adhvaryu_learning_2014, dupas_short-run_2014}. The returns to adoption could be an additional explanation \cite{banerjee_improving_2010, adhvaryu_learning_2014}. The unobserved characteristics of adopters may account substantially for the value assigned to protective or curative measures \cite{geoffard_rational_1996, chan_learning_2006, lamiraud_therapeutic_2007}. For instance, in the case of vector-borne diseases such as malaria, unobserved reasons for not using insecticide treated nets (such as personal beliefs, perceived temperature, smell, number of mosquitoes or comfort) have often been presented as evidence to explain the low adoption or use of nets. However health care researchers have traditionally regarded such reasons as the result of irrational, misinformed or subjective behavior.

In this work, we study a related process at the heart of community-based healthcare strategies, with an emphasis on vector-control programs\footnote{Vector control is any method to limit or eradicate the mammals, birds, insects or other arthropods which transmit disease pathogens.}: community involvement and households' protective behaviors in the context of a public intervention. We focus on mechanical elimination techniques which, in this study, refer primarily to the elimination of breeding sites to reduce the mosquito population around the house. Mechanical elimination is among the cheapest interventions described above and can therefore result in high monetary returns. Such methods have been recommended by the World Health Organization (WHO) for the control of arboviral diseases or malaria (for exophagic/exophilic \textit{aedes} or \textit{anopheles} that are best controlled through the destruction of breeding sites) and can be implemented either by an external agency or directly by households in their private dwellings \cite{organization_dengue_2009}. Mechanical control has also been modelled in \cite{dumont_temporal_2008,dumont_vector_2010,dumont_mathematical_2012} to study its impact on the \textit{Aedes albopictus} population and on the epidemiological risk. We argue that eliminating breeding sites is a choice made by the household and, consequently, adoption rates reflect the household’s trade-off between perceived costs and observed efficacy. In addition, we show that a mechanical elimination intervention piloted by an external agent might act as a substitute for private protection rather than a complement. Furthermore, if the intervention induces a psychological effect (a perceived improvement in safety or well-being), community (aggregate) protection level might be lower in the intervention group. As a result, those who received the intervention may be worse off than if they had not received it. 

Based on previous works, done by the authors, our study focuses on \textit{Aedes} species, and in particular \textit{Aedes albopictus} \cite{benedict_spread_2007}, also called the tiger mosquito. It is particularly threatening due to its potential for transmitting a wide range of arboviruses, including dengue, chikungunya, yellow fever, and several other types of encephalitides  \cite{gratz_critical_2004, angelini_chikungunya_2008, delatte_blood-feeding_2010, paupy_comparative_2010}. \textit{A. albopictus} is now well established in places where the socioeconomic level of the population is high (in Southern Europe, for instance). Education levels are also high, as is knowledge of disease transmission and recommended practices specific to the elimination of larval sites. Though our study is applicable to risk from arboviral diseases\footnote{There are no specific antiviral medicines or vaccines against Chikungunya and Dengue.} and \textit{Aedes} mosquitoes, it holds implications for other vector-borne diseases control, like malaria. 

Numerical simulations of our model focus on Réunion Island. Réunion is one of the places in the world that has experienced a number of epidemics due to the favourable environment it provides for the mosquito species to thrive. Past outbreaks of chikungunya and dengue prompted authorities on the island to implement strategies to control mosquito density. Indeed, since the resurgence of dengue in 2004, and the major chikungunya outbreak in 2005-2006, Health authorities have set up entomologic surveillance of \textit{Aedes albopictus} in all urban areas. This surveillance effort still continues today through the monitoring of traditional stegomyia indices at immature stages (i.e. Container Index, House Index, Breteau Index)\footnote{The house index is defined as the percentage of houses infested by larvae and/or pupae. The container index is defined as the percentage of water-holding containers with active immature stages of mosquitoes. The Breteau index is defined as the number of positive containers per 100 houses, a positive container being one that contains larval and/or pupal stages of mosquito.} as are used in other control programmes \cite{pierre_dengue_2006}.  \textit{Aedes albopictus} remains the main target of the work of the LAV (\textit{Lutte anti-vectorielle}), a service which is organised by the Regional Health Agency in Réunion.  The vector-control strategy integrates five core activities: vector surveillance, environmental, mechanical, and chemical control (larvicide being applied rarely), and public health education campaigns. Vector control services also undertake the early detection and treatment of cases of arboviral infection to prevent the spread of new epidemics. 

Given the investment of both financial and human resources toward the control of \textit{Aedes albopictus} and the observed rise of \textit{Aedes albopictus} density in Réunion from 2006 to 2011 despite public action \cite{boyer_spatial_2014}, we begin by modeling the household decision to eliminate larval gites before providing an experimental test of the theory. While it fits to the most recent bio-mathematical models applied to vector-control, the model differs from a mere bio-mathematical approach and accounts for selfish externalities and human behaviors, reflecting the trade-off between perceived costs and observed efficacy. We show that the elimination of larval habitats by an external public agent may increase the mosquito population weeks (or months) after the intervention. This situation corresponds to one of the numerical simulations of the model (that uses realistic parameters on the island).  

The rest of this paper is organized as follows.  In section \ref{sec:2} we present our Mosquito model and provide some qualitative results. In section \ref{sec:3}, we include Individual behaviors in the entomological model. Then, in section in \ref{sec:4}, we provide numerical simulations and we discuss several scenarii. Section \ref{sec:5} concludes.

\section{The Mosquito Model} \label{sec:2}

Before combining individual behaviors with an entomological model, we will first consider the following mosquito model, based on models developed and studied in \cite{dumont_mathematical_2012} (see also the related epidemiological models \cite{dumont_temporal_2008, dumont_vector_2010,rodrigues_dengue_2012, manore_comparing_2014}, and references therein)
	
 \begin{equation}
\left\{ \begin{array}{l}
\dfrac{dL_{v}}{dt}=rbA_{v}\left(1-\dfrac{L_{v}}{K_{v}}\right)-\left(\nu_{L}+\mu_{L}\right)L_{v},\\
\dfrac{dA_{v}}{dt}=\nu_{L}L_{v}-\mu_{v}A_{v}, \\
L_v(0)=L_0, \\
A_v(0)=A_0,
\end{array}\right.
\label{systemA_N}
    \end{equation}
where $A_v$ represents the adult mosquito population, and $L_v$, the "aquatic" (including eggs, larvae, and pupae) population. The (biological) parameters of the models are described as follows: $r$ is the sex ratio, $b$ is the mean number of eggs laid by a female mosquito per day that have emerged as larvae, $K_{v}$ is the maximal breeding capacity, $1/(\nu_{L}+\mu_{L})$ is the mean time a mosquito stays in the aquatic stage, $\mu_{L}$ is the aquatic daily death-rate, $\mu_{v}$ is the female mosquito mean death-rate per day.

The right-hand side of system (\ref{systemA_N}) is a continuously differentiable map ($\mathcal{C}^{1}$). Then, by the Cauchy-Lipschitz theorem, system (\ref{systemA_N}) provides a unique maximal solution.
System (\ref{systemA_N}) is biologically well posed: if the initial data are in $\mathbb{R}^2_+$, then the solution stays in $\mathbb{R}^2_+$: $L_v=0$, and $A_v=0$ are vertical and horizontal null lines respectively. Thus, no trajectory can cut these axes.
In fact, it is straighforward to show that the compact 
$$
\mathcal{K}=\left\{ (L_v,A_v)\in \mathbb{R}^2_+: L_v\leq K_v,A_v\leq \dfrac{\nu_L}{\mu_v} K_v \right\}
$$
is positively invariant by (\ref{systemA_N}).
 
Let us now give some qualitative results on system (\ref{systemA_N}). Let \begin{equation}
\mathcal{N}=\frac{\nu_{L}rb}{(\nu_{L}+\mu_{L})\mu_{v}},
\label{N}
\end{equation}
the basic offspring number\footnote{The basic offspring number represents the average number of offsprings produced over the lifetime of an individual (under ideal conditions)}. First, we can show that there exist two possible equilibria:
\begin{itemize}
\item a trivial equilibrium $\mathbf{0}=(0,0)$
\item a positive equilibrium $\mathbf{E}=(L_v^*,A_v^*)$, when $\mathcal{N}>1$, defined as follows:
 \begin{equation}
\left\{
\begin{array}{l}
L_v^{*},=\left(1-\dfrac{1}{\mathcal{N}}\right)K_{v}, \\
A_v^*=\dfrac{\nu_{L}}{\mu_{v}}\left(1-\dfrac{1}{\mathcal{N}}\right)K_{v}.
\end{array}
\right.
\label{Eq}
     \end{equation}
     $\mathbf{E}$ belongs to $\mathcal{K}$, and, thus, is biologically realistic, i.e. it is positive and bounded.
\end{itemize}
The dynamic of model (\ref{systemA_N}) can be summarized by the following proposition:

\begin{proposition} Assume that $(L_0,A_0)\in \mathcal{K}$.
\label{propo1}
\begin{itemize}
\item When $\mathcal{N}\leq 1$, the trivial equilibrium $\mathbf{0}$ is globally asymptotically stable, which means that the mosquito population will dwindle until extinction, whatever the initial population.
\item When $\mathcal{N}>1$, the trivial equilibrium is unstable and the  positive equilibrium $\mathbf{E}$ is globally asymptotically stable, which means that the mosquito population persists.
\end{itemize}
\end{proposition}     

\begin{proof}
We will use some results related to cooperative systems. These type of systems are very important and well-known in Biology and Ecology, for which many useful theoretical results have been proved \cite{smith_monotone_2008}. A general definition is

\begin{definition}
\label{defCoop} The system $\dot{x}=f(x)$ is called \emph{cooperative} if for every $i,j\in\{1,2,...,n\}$ such that $i\neq j$ the function $%
f_i(x_1,...,x_n)$ is monotone increasing with respect to $x_j$.
\end{definition}

It is also easy to verify that system (\ref{systemA_N}) is a cooperative system in $\mathcal{K}$. To show both GAS results, we will use a theorem proved in \cite[Theorem 6, page 376]{anguelov_mathematical_2012}.
\begin{itemize}
\item When $\mathcal{N}\leq 1$,  system (\ref{systemA_N}) has only the trivial equilibrium $\mathbf{0}$. Thus taking $\mathbf{a}=\mathbf{0}$ and $\mathbf{b}=(K_v,\dfrac{\nu_L}{\mu_v \mathcal{N}}K_v)$, we have $f(\mathbf{a})=0$ and $f(\mathbf{b})<0$. It follows from Theorem 6 (\cite{anguelov_mathematical_2012}), that $\mathbf{0}$ is globally asymptotically stable on $[\mathbf{0},\mathbf{b}]$, hence on $\mathcal{K}$, when  $\mathcal{N} \leq 1$.
\item When $\mathcal{N}>1$, there exists $\varepsilon>0$ such that $\mathcal{N}>1+\varepsilon$. Let $L_{\varepsilon}$ sufficiently small such that
$$
\begin{array}{l}
L_{\varepsilon}\leq \varepsilon,\\
A_{\varepsilon}=\dfrac{\nu_L\left(1+\varepsilon \right)}{\mu_v \mathcal{N}} L_{\varepsilon} \leq \varepsilon
\end{array}
$$
Let $\mathbf{a}_\varepsilon=(L_{\varepsilon},A_\varepsilon)^T$. Then, from the right-hand side of (\ref{systemA_N}) and the fact that $\mathcal{N}>1+\varepsilon$ and $K_v>>1$, we deduce
\begin{equation}
f(\mathbf{a}_\varepsilon)\geq \left(
\begin{tabular}{c}
$(\nu_{L}+\mu_{L})\varepsilon\left(1-\dfrac{1+\varepsilon}{K_v}\right)L_\varepsilon$ \\
$\nu_{L}\left(1-\dfrac{1+\varepsilon}{\mathcal{N}}\right) L_\varepsilon$%
\end{tabular}
\right)\ > \ \mathbf{0}.
\end{equation}
Hence it follows from Theorem 6 \cite{anguelov_mathematical_2012} that equilibrium  $\mathbf{E}=(L_v^*,A_v^*)^T$ is globally asymptotically stable on $[\mathbf{a}_\varepsilon,\mathbf{b}]$. Since $\mathbf{a}_\varepsilon$ can be selected to be smaller than any $x>\mathbf{0}$, we have that $\mathbf{E}$ is asymptotically stable on $\mathcal{K}$ with basin of
attraction $\tilde{\mathcal{K}}=\mathcal{K}\setminus\mathbf{0}$. This also implies that $\mathbf{0}$ is unstable and conclude the proof.
\end{itemize}
\end{proof}

It is also obvious that the mosquito population is related to the maximum larvae or aquatic capacity $K_{v}$. Thus, reducing $K_{v}$  will decrease the mosquito population. This explains why mechanical control (i.e. the action of reducing all potential breeding sites) is very important. Of course it is well known that $K_v$ may change according to the environment \cite{dufourd_impact_2013}.
 
\section{Economic entomological model with protection}
\label{sec:3}
For the rest of the section, we assume that $\mathcal{N}>1$. Let us now consider that individuals use either mechanical control or do not use mechanical control. Assuming that a proportion $H$ of people agree to consider mechanical control, we suppose that the efficiency of the mechanical control is related to $H$ such that, the decrease in $K$ becomes more important as more people practice mechanical control. Thus, when mechanical control occurs, the aquatic carrying capacity decays according to the following law: 
	 \begin{equation}
	K_v=(1-\gamma(H))K_{v},
      \end{equation}
where $\gamma$ is defined such that $\gamma \in[0,1]$ is a non-decreasing function, with $\gamma(0)=0$ and $0<\gamma(1)<1$ , which means that we cannot destroy all breeding sites, even if the whole population were to participate, $H=1$.
The efficacy $\gamma$ can depend on several parameters like environmental parameters, such as rainfall or the quality of mechanical control. Thus $\gamma$ can be defined as a linear function, like
	 \begin{equation}
\gamma(x)=ax, \quad \mbox{ with }a\in]0,1[ 
      \end{equation}
or as a sigmoid-shape function
	 \begin{equation}
 \gamma(x)=\frac{ax}{1+ax},\quad\mbox{with }a>0,
       \end{equation}
where $a$ is a parameter that can be related to the rate of efficacy.

Note also that the larvae capacity $K_v$   does not remain constant, even after intervention. Thus, $K_v$ tries to reach $K_{\max}$, the maximal larvae capacity. The carrying capacity varies according to environmental parameters, and we don't know exactly how to model it. That is why several empirical models could be considered. Among them, we choose one of the simplest. Let $r_K$ a parameter related to the growth rate. For instance, if we assume that the environment is favorable (due to regular rainfall, for instance), then $r_K$ can be large such that $K_v$ grows rapidly to $K_{\max}$, that can be choosen very large. In contrast, when the environment is not favorable (drought or dry season), then $r_K<<1$ (and even equal to $0$) and $K_{\max}$ can be choosen very low. Altogether, we can consider the following growth equation that features the previous properties according to the sign of $r_K$:
\begin{equation}
\left\{ 
\begin{array}{l}
\dfrac{dK_v}{dt}=r_{K}\left(K_{\max}-K_v\right),\\
K_v(0)=K_{0}
\end{array}\right.
\label{KGrowth}
\end{equation}
Thus, assuming that $r_K>0$, since $K_V(t)=K_{\max}-\left(K_0-K_{\max}\right)e^{-r_Kt}$, $K_v$ will converge (rapidly or not, depending on $r_K$) to the maximum larvae capacity, $K_{\max}$.
\begin{remark}
According to environmental parameters, like temperature and rainfall, $K_{\max}$ and $r_K$ can be chosen time-dependant, which implies that
$$
K(t)=K_{\max}(t)+(K_0-K_{\max}(t))e^{-\int_0^tr_K(\tau)d\tau}.
$$
\end{remark}
If mechanical control occurs on day $t_m$, then we consider this control as ``instantaneous'' (say, it is done in one day). This can be modeled as follows:
\begin{equation}
K_v(t_m^+)=(1-\gamma(H(t_m))K_{v}(t_m).
\end{equation}
Thus assuming that mechanical control starts on day $t_0$ and occurs every $\tau$ days, then, we obtain the following impulsive differential equation
\begin{equation}
\left\{ \begin{array}{l}
\dfrac{dK_v}{dt}=r_{K}\left(K_{\max}-K_v\right),\\
\Delta K_v(t_0+n\tau)=-\gamma(H(t_0+n\tau))K_{v}(t_0+n\tau), \quad n=0,1,2,..., \\
K_v(0)=K_{0},
\end{array}\right.
\label{Impulsive}
\end{equation}
where $\Delta x(t)=x(t^+)-x(t)$.
According to the theory of impulsive differential equations \cite{bainov_impulsive_1993}, equation (\ref{Impulsive}) is well defined and has a unique positive solution. If $H$ is assumed to be constant, i.e. $H(t)=H_0$ then, using straightforward computations, it is possible to show the following results. 

\begin{proposition} 
\label{propPeriodiqueK}
Equation (\ref{Impulsive}) admits the following periodic equilibrium $K_{v,per}$ :
\begin{equation}
K_{v,per}(t)=\left(1-\dfrac{\gamma(H_0)e^{-r_K(t-n\tau)}}{1-(1-\gamma(H_0))e^{-r_K\tau}} \right)K_{\max}, \quad t\in [n\tau, (n+1)\tau[ \mbox{ and } n=0,1,2,... 
\end{equation}
\end{proposition}
Then, setting $y=K_v-K_{v,per}$ it is straightforward to show that $y$ is solution of
\begin{equation}
\left\{ \begin{array}{l}
\dfrac{dy}{dt}=-r_{K}y,\\
\Delta y(t_0+n\tau)=0, \quad n=0,1,2,..., \\
y(0)=K_{0}-K_{v,per}(0).
\end{array}\right.
\end{equation}
Obviously, $y$ goes to $0$ as $t$ goes to $+\infty$. Thus, we deduce
\begin{proposition}
\label{asymptoticK}
$K_v$ solution of (\ref{Impulsive}) converges globally asymptotically to $K_{v,per}$
\end{proposition}
Set 
\begin{equation}
C=r_K\dfrac{\gamma(H_0)}{1-e^{-r_K\tau}}.
\end{equation}
From the previous propositions, we can show
\begin{proposition}
Assume that $H$ converges to $H_0$. If $\mathcal{N}$, $\gamma(H_0)$ and $\tau$ are choosen such that
\begin{itemize}
\item 
\begin{equation}
\mathcal{N}> \left(1+\frac{\mu_v}{\nu_L}\dfrac{1}{1-\gamma(H_0)}C\right)\left(1+\frac{1}{\nu_L+\mu_L}\dfrac{1}{1-\gamma(H_0)}C\right),
\label{condPer}
\end{equation}
then (\ref{systemA_N})-(\ref{Impulsive}) admits a unique periodic equilibrium which attracts all initial conditions in $\mathcal{K}$.
 \item 
\begin{equation}
\mathcal{N}\leq  \left(1+\frac{\mu_v}{\nu_L}C\right)\left(1+\frac{1}{\nu_L+\mu_L}C\right),\label{condZero}
\end{equation}
  then the solution of system (\ref{systemA_N})-(\ref{Impulsive}) converges to the trivial equilibrium $\mathbf{0}$. 
\end{itemize}
\end{proposition} 

\begin{proof}
Using proposition \ref{propPeriodiqueK}, we know that $K_v$ converges to a periodic equilibrium. Thus assuming that $t$ is sufficiently large, system (\ref{systemA_N}) becomes a periodic cooperative system of period $\tau$:
 \begin{equation}
\left\{ \begin{array}{l}
\dfrac{dL_{v}}{dt}=rbA_{v}\left(1-\dfrac{L_{v}}{K_{v,per}(t)}\right)-\left(\nu_{L}+\mu_{L}\right)L_{v},\\
\dfrac{dA_{v}}{dt}=\nu_{L}L_{v}-\mu_{v}A_{v}.
\end{array}\right.
\label{systemA_N_Per}
\end{equation}
The previous system is not really practical to handle, in particular because the maximal larvae capacity is now time dependant. To circumvent this difficulty, we make the following change of variables:
\begin{equation}
u=\frac{L_v}{K_{v,per}} \qquad \mbox{and} \qquad v=\dfrac{\mu_v}{\nu_L}\frac{A_v}{K_{v,per}},
\label{changement_variable}
\end{equation}
such that $(u,v)\in [0,1]^2$ and
$$
\begin{array}{ll}
\dfrac{du}{dt}&=\dfrac{1}{K_{v,per}} \dfrac{dL_{v}}{dt}-u\times \dfrac{1}{K_{v,per}} \dfrac{dK_{v,per}}{dt}\\
&=rb\dfrac{\nu_L}{\mu_v}v(1-u)-\left(\nu_{L}+\mu_{L}+r_K\left(\dfrac{K_{\max}}{K_{v,per}}-1\right) \right) u
\end{array}
$$
and
$$
\begin{array}{ll}
\dfrac{dv}{dt}&=\left(\dfrac{1}{K_{v,per}} \dfrac{dA_{v}}{dt}-v\times \dfrac{1}{K_{v,per}} \dfrac{dK_{v,per}}{dt}\right) \\
&=\nu_L \left( u-\left(1+\dfrac{\mu_v}{\nu_L}r_K\left(\dfrac{K_{\max}}{K_{v,per}}-1\right) \right) v\right),
\end{array}
$$
which leads to the system
 \begin{equation}
\left\{ \begin{array}{l}
\dfrac{du}{dt}=rb\dfrac{\nu_L}{\mu_v}v(1-u)-\left(\nu_{L}+\mu_{L}+r_K\left(\dfrac{K_{\max}}{K_{v,per}}-1\right) \right) u,\\
\dfrac{dv}{dt}=\nu_L\left( u-\left(1+\dfrac{\mu_v}{\nu_L}r_K\left(\dfrac{K_{\max}}{K_{v,per}}-1\right) \right) v\right).
\end{array}\right.
\label{systemA_N_PerBIS}
\end{equation}
In fact system (\ref{systemA_N_Per}) enters the family of periodic cooperative systems with a concave nonlinearity \cite{smith_cooperative_1986, Jiang_1993}. So, let us consider a periodic cooperative system of $n$ differential equations 
\begin{equation}
\dot{x}=F(t,x),
\label{Fgeneral}
\end{equation}
with concave nonlinearities. Since $A(t)$ is a $n\times n$ continuous matrice in $\mathbb{R}$, $\tau$-periodic in $t$, we denote
$$
\overline{a}_{ij}=\max_{0\leq t \leq \tau} a_{ij}(t), \quad \underline{a}_{ij}=\min_{0 \leq t\leq \tau} a_{ij}(t)
$$
and set
$$
\overline{A}=(\overline{a}_{ij}) \qquad \underline{A}=(\underline{a}_{ij}).
$$
Thus 
$$
\underline{A} \leq A(t) \leq \overline{A} \quad \mbox{for}\quad 0\leq t \leq  \tau.
$$
Set $p$ a positive real. For reader's convenience, let us recall \cite[Theorem 5.5, page 230]{Jiang_1993}.

\begin{theorem}
\label{top}
Let $F(t,x)$ be continuous in $\mathbb{R}\times [0,p]^2$, $\tau$-periodic in $t$ for a fixed $x$ and assume $D_xF(t,x)$ exists and is continuous in $\mathbb{R}\times [0,p]^2$. Assume that all solutions are bounded in $[0,1]^2$ and $F(t,0) =0$. Assume
$$
\frac{\partial F_i}{\partial x_j}\geq 0, \qquad (t,x)\in \mathbb{R}\times [0,p]^2,
$$
and $A(t)=D_xF(t,0)$ is irreducible for any $t\in \mathbb{R}$,
$$
(C)\qquad \mbox{if }0<x<y, \qquad \mbox{then } \qquad D_xF(t,x)>D_xF(t,y).
$$
Then
\begin{enumerate} 
\item If all principal minors of $-\overline{A}$ are nonnegative, then $\lim_{t\rightarrow \infty} x(t)=0$ for every solution of (\ref{Fgeneral}) in $[0,p]^2$.
\item If $-\underline{A}$ has at least one negative principal minor, then (\ref{Fgeneral}) possesses a unique positive $\tau$-periodic solution which attracts all initial conditions in $]0,p]^2$.
\end{enumerate}
\end{theorem}
In our case, we have $p=1$ and
$$
F(t,x)=\left(\begin{array}{ll}
rb\dfrac{\nu_L}{\mu_v}x_2\left(1-x_1\right)-\left(\nu_{L}+\mu_{L}+r_K\left(\dfrac{K_{\max}}{K_{v,per}}-1\right) \right)x_1 \\
\nu_L\left( x_1-\left(1+\dfrac{\mu_v}{\nu_L}r_K\left(\dfrac{K_{\max}}{K_{v,per}}-1\right) \right) x_2\right)
\end{array}
\right).
$$
Clearly $F$ is continuous and $\tau$-periodic. Moreover
$$
\partial_xF(t,x)=\left(\begin{array}{cc}
-\left(rb\dfrac{\nu_L}{\mu_v}x_2+\nu_{L}+\mu_{L}+r_K\left(\dfrac{K_{\max}}{K_{v,per}}-1\right) \right) & rb\dfrac{\nu_L}{\mu_v}\left(1-x_1\right) \\
\nu_L  & -\nu_L \left(1+\dfrac{\mu_v}{\nu_L}r_K\left(\dfrac{K_{\max}}{K_{v,per}}-1\right) \right) 
\end{array}
\right).
$$
From which we deduce that $\dfrac{\partial F_i}{\partial x_j}\geq 0$ for all $(x_1,x_2)\in [0,1]^2$. When $0<x<y$, we have obviously
$$
\partial_xF(t,y)<\partial_xF(t,x).
$$
Finally
$$
A=\partial_xF(t,0)=\left(\begin{array}{cc}
-\left(\nu_{L}+\mu_{L}+r_K\left(\dfrac{K_{\max}}{K_{v,per}}-1\right) \right) & rb\dfrac{\nu_L}{\mu_v} \\
\nu_L  & -\nu_L \left(1+\dfrac{\mu_v}{\nu_L}r_K\left(\dfrac{K_{\max}}{K_{v,per}}-1\right)
\right) \end{array}
\right).
$$
Clearly, $A$ is irreducible for all $t$. Using Proposition (\ref{propPeriodiqueK}), we have
$$
 \dfrac{(1-\gamma(H_0))(1-e^{-r_K\tau})}{1-(1-\gamma(H_0))e^{-r_K\tau}}K_{\max} \leq K_{v,per}(t) \leq \dfrac{1-e^{-r_K\tau}}{1-(1-\gamma(H_0))e^{-r_K\tau}}K_{\max}.
$$ 
Thus, we deduce
$$
-\underline{A}=\left(\begin{array}{cc}
\left(\nu_{L}+\mu_{L}+r_K\dfrac{\gamma(H_0)}{(1-\gamma(H_0))(1-e^{-r_K\tau})} \right) & -rb\dfrac{\nu_L}{\mu_v} \\
-\nu_L  & \nu_L \left(1+\dfrac{\mu_v}{\nu_L}r_K\dfrac{\gamma(H_0)}{(1-\gamma(H_0))(1-e^{-r_K\tau})}
\right) \end{array}
\right)
$$
and
$$
-\overline{A}=\left(\begin{array}{cc}
\left(\nu_{L}+\mu_{L}+r_K\dfrac{\gamma(H_0)e^{-r_K\tau}}{1-e^{-r_K\tau}} \right) & -rb\dfrac{\nu_L}{\mu_v} \\
-\nu_L  & \nu_L \left(1+\dfrac{\mu_v}{\nu_L}r_K\dfrac{\gamma(H_0)e^{-r_K\tau}}{1-e^{-r_K\tau}}
\right) \end{array}
\right)
$$
First, according to Theorem \ref{top}$_1$, we have to study all principal minors of $-\overline{A}$. The diagonal terms are positive, and a straighforward computation show that $det(-\overline{A})\leq 0$ if following inequality is verified
$$
\mathcal{N}\leq  \left(1+\frac{\mu_v}{\nu_L}e^{-r_K\tau}C\right)\left(1+\frac{1}{\nu_L+\mu_L}e^{-r_K\tau} C\right).
$$
Following Theorem \ref{top}$_2$, since all diagonal terms of $-\underline{A}$ are positive, we need to verify that $det(-\underline{A})<0$. A straigforward computation shows that we have to verify
$$
\mathcal{N}> \left(1+\frac{\mu_v}{\nu_L}\dfrac{1}{1-\gamma(H_0)}C\right)\left(1+\frac{1}{\nu_L+\mu_L}\dfrac{1}{1-\gamma(H_0)}C \right).
$$
Thus, applying Theorem \ref{top}, using the change of variables (\ref{changement_variable}), and the fact that $K_{v,per}$ is a $\tau$- periodic function, we deduce the desired results.
\end{proof}

\subsection{The Model with Individual behaviors}
Now, we aim to define $H$. Taking into account that we have individuals that will make mechanical control, $m=1$, and individuals that will not make mechanical control, $m=0$, we only consider two status for people: being bitten (or bothered) by mosquitoes, $\sigma(b)=b$, and not being bitten (or bothered), $\sigma(b)=nb$. Then, the probability of being bitten, at least once, without mechanical control, is:
		 \begin{equation} \label{eq:proba0}
\pi=P(b(m)=1/m=0).
        \end{equation}
Similarly, we can define the following probabilities:
		 \begin{equation} \label{eq:proba1}
\left\{ \begin{array}{c} 
P(b(m)=0/m=0)=1-\pi,\\
P(nb(m)=1/m=1)=1.
\end{array}\right.
        \end{equation}
\begin{remark}
\label{remarkpiqure}
For the numerical simulations, we have to define $\pi$, the probability of being bitten, at least $k$ times. We will consider a Poisson law, $\mathcal{P}(\lambda)$, where $\lambda$ depend on the number of mosquitoes per human, i.e. $\lambda=\dfrac{A_{v}}{N_h}$. Thus
$$
\pi=P(X\geq k)=1-F_{P}(k,\lambda),
$$
where $F_P$ is the cumulative Poisson distribution, $k$, the number of bites, is choosen according to the level of tolerance of the population to mosquito bites.
\end{remark}

\bigskip

Thus, taking into account that mechanical control has a direct cost (spend a couple of hours per week to clean the garden) and an indirect cost (pay people to make mechanical control: in general people paid by a public agency, and thus indirectly by taxes), we have the following maximization program:
		 \begin{equation}
max_{m}E[u(\sigma(m))]-\kappa W m,
          \end{equation}
where $u(b)$ and $u(nb)$  are the utility levels attached to the bitten status, with $0<u(b)<u(nb)$. $W$ is the marginal utility and is supposed to be constant, and $\kappa$  is the cost of mechanical control. $\kappa$ may depend on two variables $s_{l}$, the level of local interventions (by people themselves) and $s_{e}$, the level of external interventions (by ARS agents), that is:

\begin{equation}
\left\{ \begin{array}{c}
\kappa(s_{l},s_{e})=\kappa_{0}s_{l}+\kappa_{1}s_{e},\\
\mbox{where }s_{e}\in\left\{ 0,1\right\} ,
\end{array}\right.
\label{costfunction}
\end{equation}
$\kappa_0$ and $\kappa_1$ are constants related to the cost of each intervention. 


Therefore, an individual will partake in mechanical control if $\kappa W$ is lower than the expected utility loss related with the risk of being bitten without mechanical control:
\begin{equation}
E\left[u(\sigma(1)-u(\sigma(0))\right]\geq\kappa W
\end{equation}

The expected utility can be estimated using (\ref{eq:proba0}) and (\ref{eq:proba1}). Altogether, an individual will use mechanical control if $\pi$   is above a certain threshold, that is:
\begin{equation}
\pi\ge\frac{\kappa_0s_lW}{u(nb)-u(b)} \qquad \mbox{mechanical control is utilized only by local people, i.e. $s_e=0$}
\label{mlonly}
\end{equation}
and
\begin{equation}
\pi\ge\frac{\left(\kappa_0s_l+\kappa_1\right)W}{u(nb)-u(b)} \qquad \mbox{mechanical control is utilized by local people and ARS agents, i.e. $s_e=1$.}
\label{mlandars}
\end{equation}
Thus, from (\ref{mlonly}), we deduce $\mathcal{S}_l$, the threshold related to $s_l$:
\begin{equation}
\mathcal{S}_l=G_0^{-1}(\pi)=\dfrac{1}{\kappa_0}\left(\dfrac{\pi\left(u(nb)-u(b)\right)}{W}\right)
\end{equation}
Therefore, using  (\ref{mlandars}), $\mathcal{S}_{l}$ becomes:
\begin{equation}
\mathcal{S}_l=G_1^{-1}(\pi)=\dfrac{1}{\kappa_0}\left(\dfrac{\pi\left(u(nb)-u(b)\right)}{W}-\kappa_1\right)
\end{equation}

Consequently, assuming that $f$ is the probability density function of $s$, the proportion of people that will finally use mechanical control can be defined as follows:
\begin{equation}
H=\int_{\mathcal{S}_m}^{+\infty}f(s)ds=1-F(G_k^{-1}(\pi)),
\label{H}
\end{equation}

where F is the cumulative distribution function related to f, and $k\in\{0,1\}$.

Thus, knowing $H$, we can solve system (\ref{Impulsive}). In particular, following the previous propositions above, we infer:

\begin{proposition}
Assume $\mathcal{N}>1$, and $H$ converges to $H_0$.
\begin{itemize}
\item If (\ref{condPer}) is true, then system (\ref{systemA_N})-(\ref{Impulsive}) will converge to a positive periodic equilibrium.
\item If (\ref{condZero}) is true, then system (\ref{systemA_N})-(\ref{Impulsive}) will converge to the zero equilibrium.
\end{itemize}
\end{proposition}  

\section{Numerical simulations}
\label{sec:4}

We now derive some simulations to better illustrate the model of the previous section and discuss some possible scenarios. 

Let us first explain the parameter values taken in our simulations.
The relationship between utility and income takes the form of a constant relative risk aversion utility function. Thus, the marginal utility of income, $W$, is given by \cite{layard_marginal_2008}:
\begin{equation}
W=e^{-\beta\log\left(Income\right)},
\end{equation}
where $\beta=1.2$, $Income=103$ euros, the daily mean salary of Réunion island\footnote{See \href{http://www.insee.fr/fr/themes/document.asp?ref_id=9865}{INSEE website} accessed July 2014.}. 
We also choose $u_{c}$ as a constant following \cite{halasa_quantifying_2014}:

\begin{equation}
u(nb)-u(b)=u_{c} \in [0,1],
\end{equation}

In particular, we will consider three different values for $u_c$, $0.3$, $0.6$ and $0.9$. Indeed, in \cite{halasa_quantifying_2014}, the authors use the EuroQol states trade-off (EuroQol-STO) to derive the mosquito-abundance-utility score (ranging between 0 and 1) by allowing residents to elicit preferences between alternative health states, instead of time, and living an average day with mosquitoes. They find that living an average day with mosquitoes in their yard and porch during the summer of 2010 was worse than living an average day with the specified comparator health conditions and diseases. The average (SD) utility based on EuroQol-STO was $0.87 \pm 0.03$, corresponding to a utility loss of $0.13$. For the cost constants $\kappa_i$, used in (\ref{costfunction}), we consider $\kappa_0=14.8$\footnote{The hourly mean salary in the considered region.}, and $\kappa_1=0$ (fully subsidized), $50$, and $100$ respectively. Finally, we compute $H$, using (\ref{H}), with $f$ following a normal distribution $\mathcal{N}(0,\dfrac{u_c}{W\kappa_0})$.

The larvae capacity recovery rate $r_{K}$ can vary throughout the year because it is generally related to environmental parameters, like rainfall. For simplicity, we consider $r_{K}$ to be constant. 

We assume that individuals implement mechanical control every week. Seventy days (10 weeks) after the start of the mechanical control, there is a field intervention from ARS such that $50\%$ of the larval capacity is destroyed in addition to the decay obtained by individuals. The ARS intervention has a cost, so more people react and partake in mechanical control during the external intervention.

According to \cite{dufourd_impact_2013}, we choose the following values for the mosquito biological parameters: $r\times b=5$, $\nu_{L}=1/15$ , $\mu_{L}=0.01$ , $\mu_{v}=0.05$, $r_K=0.05$, $K_{max}=2\times 10^6$. As initial vector we choose $(L(0),A(0),K(0))=(20000,20000,20000)$. These parameters reflect realistic conditions in Réunion.

The numerical simulations are provided using the nonstandard finite difference method, which has showed to be very effective to solve ODES systems in mathematical epidemiology and population dynamics (see for instance \cite{anguelov_nonstandard_2012, anguelov2013, anguelov_dynamically_2014} and references therein). Nonstandard methods are also suitable to solve impulsive differential equations too \cite{Dumont2015}

\begin{figure}[hbtp]
\centering
	\includegraphics[width=0.75\textwidth]{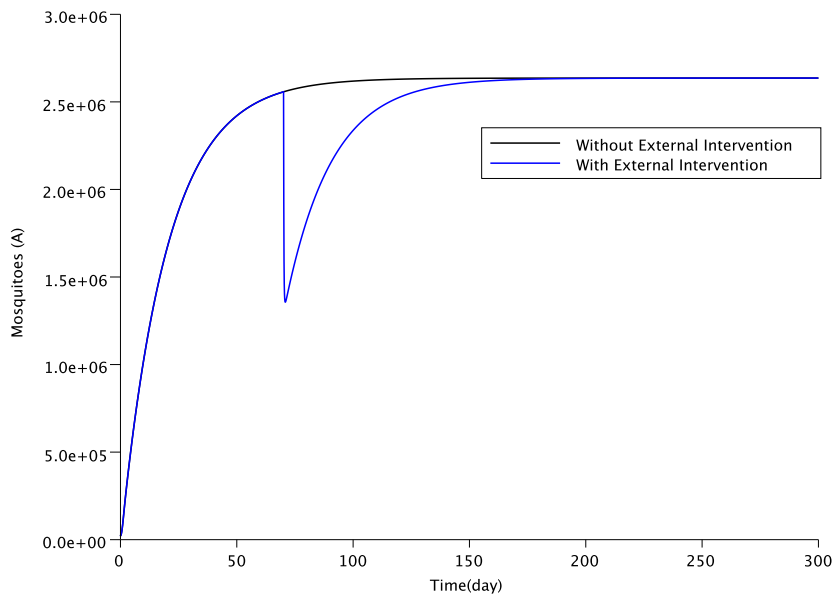} 
\caption{Baseline Scenario - Evolution of the mosquito population (with - blue line - and without - black line - External intervention  at day 70; local people don't yet make mechanical control)}
\label{fig1}
\end{figure}

In Figure \ref{fig1}, we show the dynamics of the adult mosquito with and without an external intervention. Note that the model excludes human behaviors. We start with an initially small number of larvae and adults to take into account the dynamics between September-October, when the mosquito population is small, and March-April, when the population is large. This transition is important: we start with $A_v(0)/N_h=0.1$ to $A_v(0)/N_h \approx 10$ at equilibrium, where $N_h$ is the Human population. Notably, the external intervention reduces the number of mosquitoes for only a couple of weeks. Afterward, the solution converges to the positive equilibrium. This indicates that destroying breeding sites should be regularly undertaken and maintained. 

In the following simulations, we illustrate our predictions when endogenous human behavior is incorporated in the model and consider different scenarios. As numerical simulations are globally robust for every sub-case of each scenario, we only provide a closer look at the middle sub-case of each scenario for illustration (the cost of the intervention is estimated as 53 euros per household).
\begin{itemize}

\item In scenario 1 (Figures \ref{scenario1a}, \ref{scenario1b}, \ref{scenario1c}), we consider different utility values $u_c$ ($0.3$, $0.6$, $0.9$) and different costs for the external intervention $\kappa_e$, i.e. $0$, $50$ and $100$ euros per individual (or house) treated. We also assume that people do not change their level of tolerance to mosquito bites. For instance, we may assume a mean tolerance threshold to mosquito bites equal to $3$ bites per day, that is $k=3$ (see remark \ref{remarkpiqure}).
\begin{figure}[hbtp]
\begin{center}
	\includegraphics[width=0.9\textwidth]{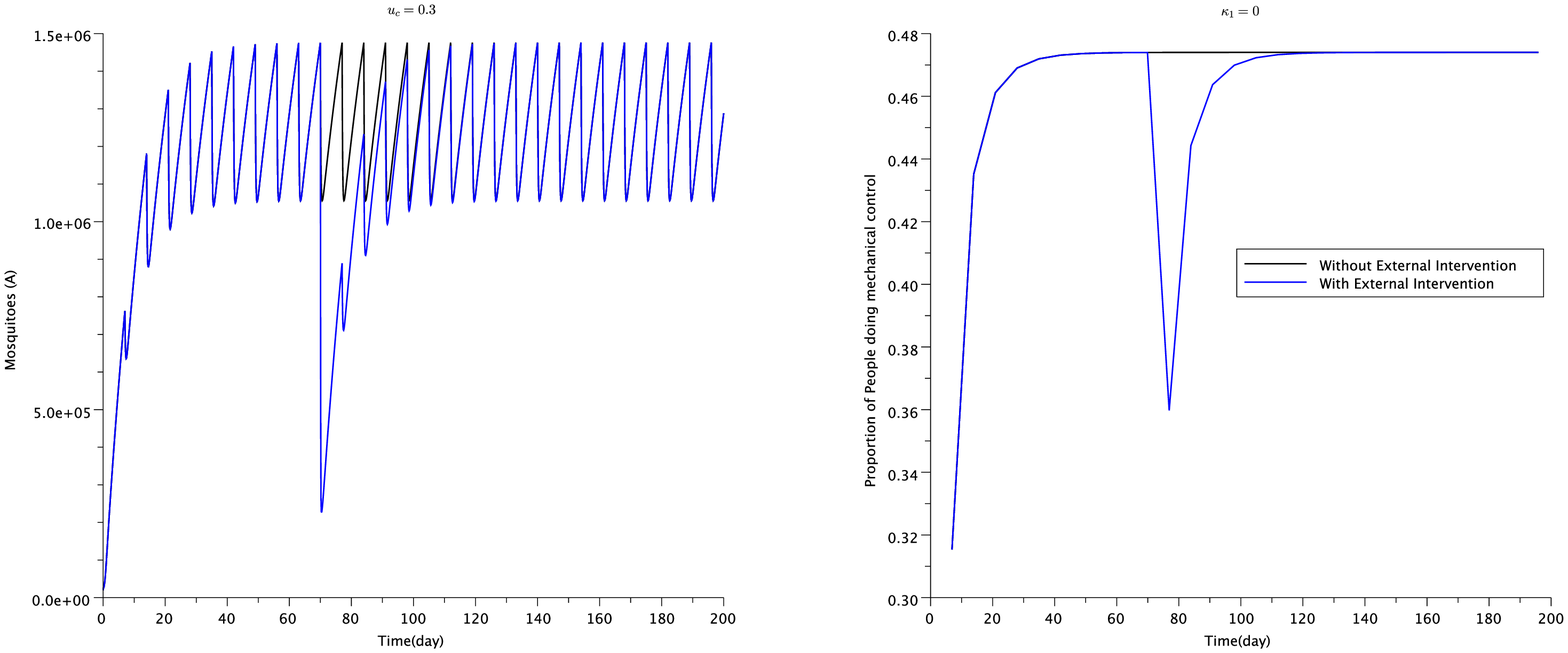} 
		\includegraphics[width=0.9\textwidth]{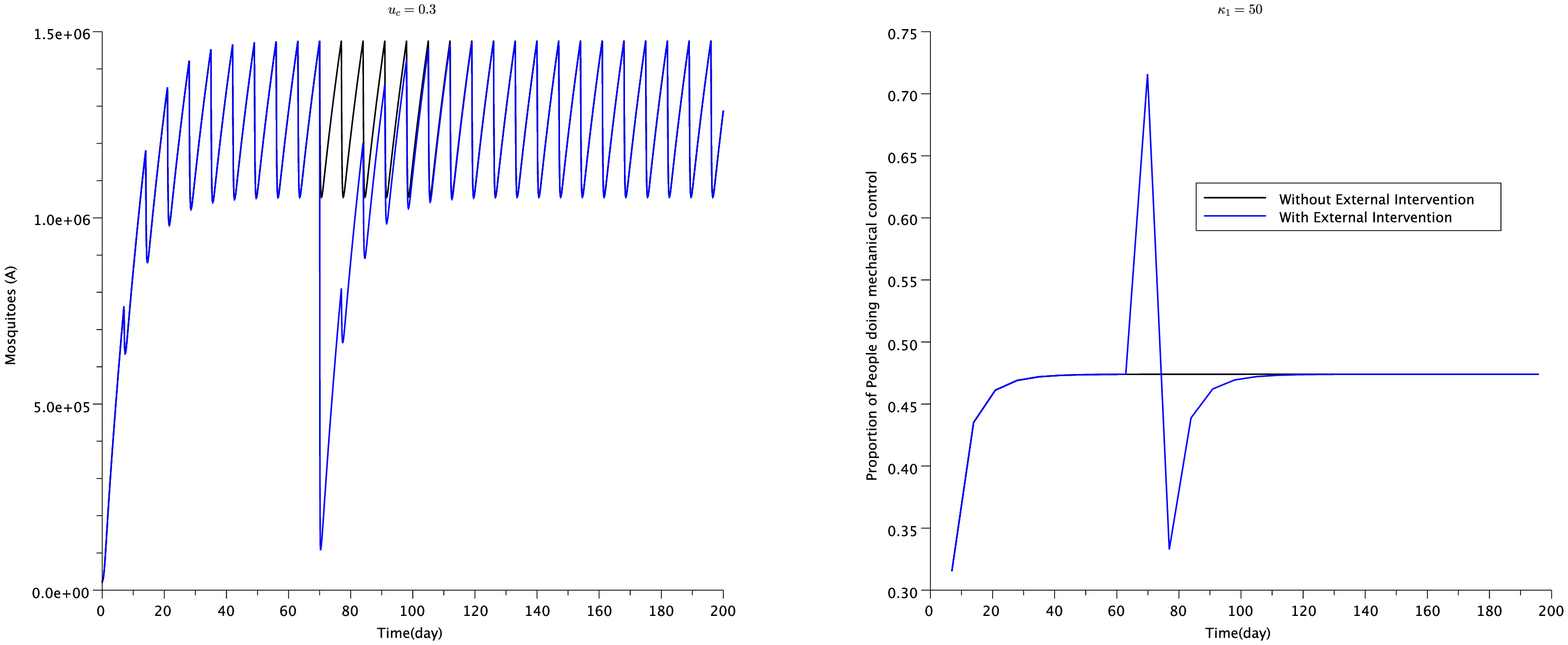} 
		\includegraphics[width=0.9\textwidth]{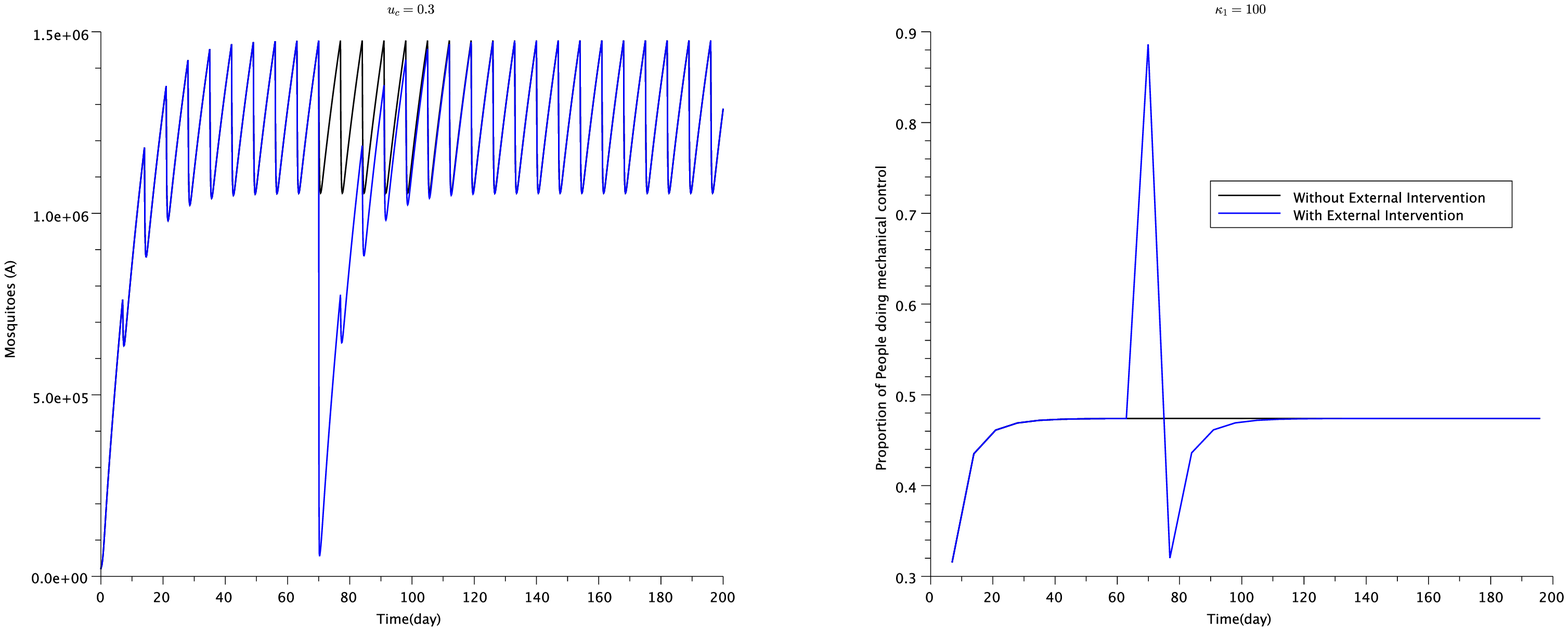} 
	
\end{center}
\caption{Scenario 1a. We consider the utility value $u_c=0.3$ with different costs for the external intervention $\kappa_e$, i.e. $0$, $50$ and $100$ euros per individual (or house) treated are considered.}
	\label{scenario1a}
\end{figure}
\begin{figure}[hbtp]
\begin{center}
	\includegraphics[width=0.9\textwidth]{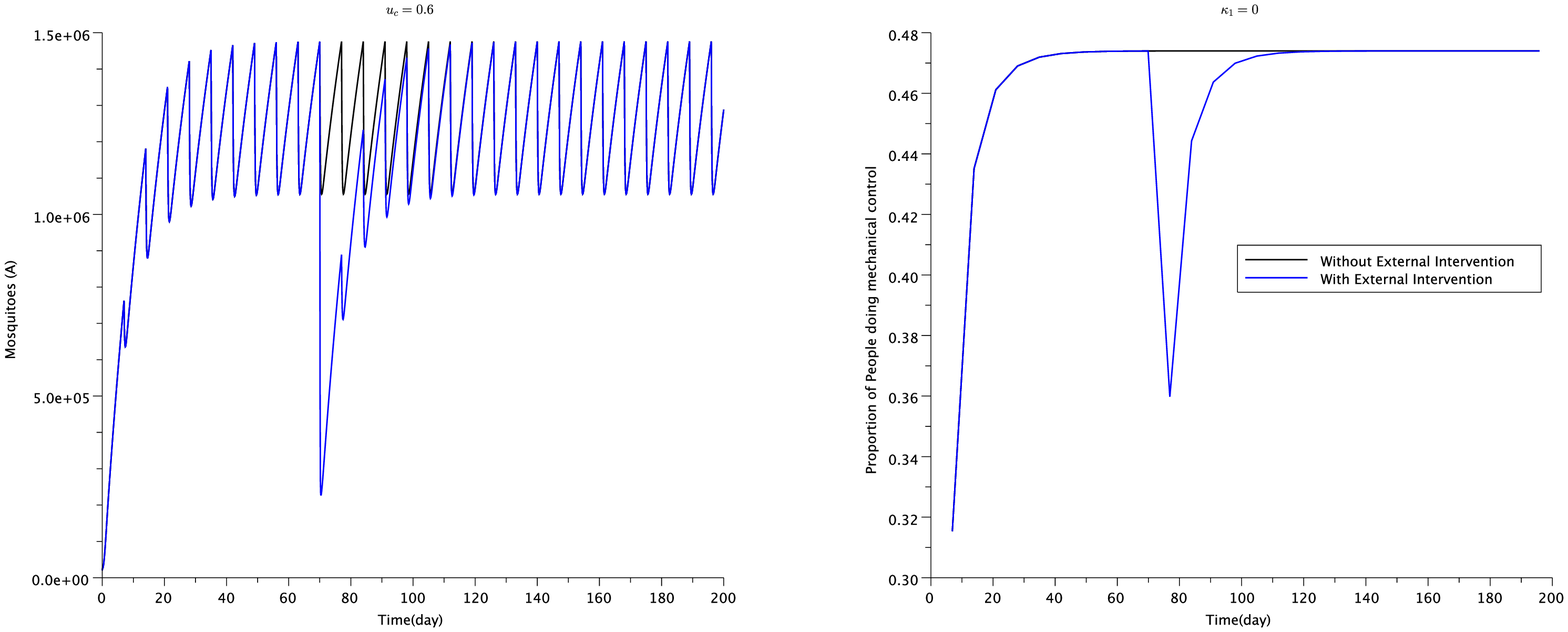} 
		\includegraphics[width=0.9\textwidth]{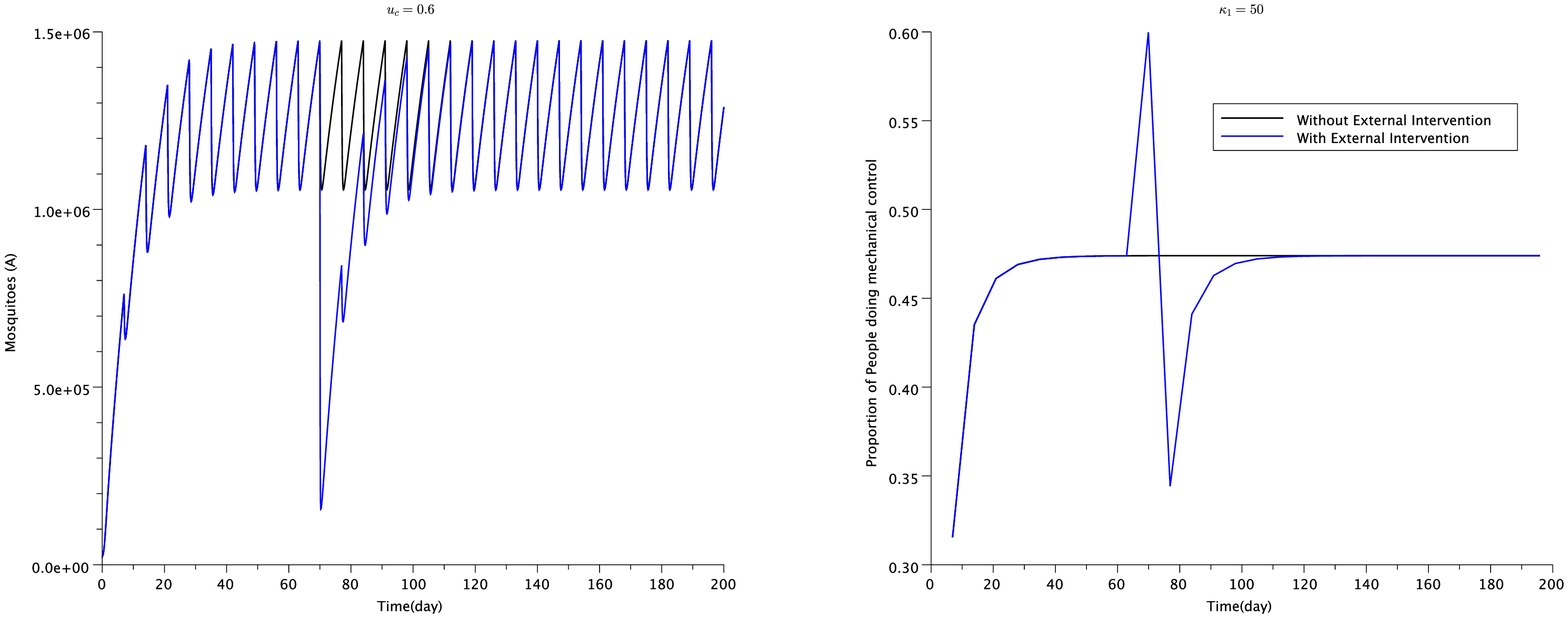} 
		\includegraphics[width=0.9\textwidth]{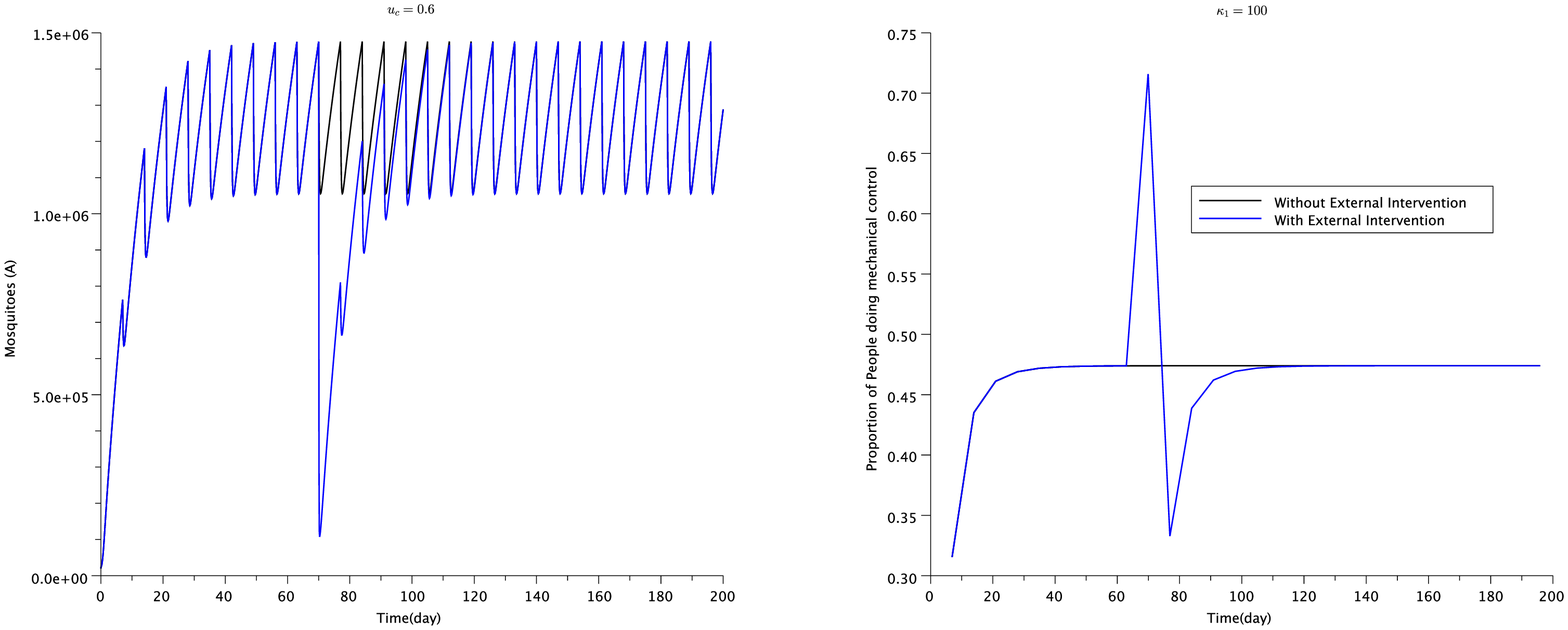} 
	
\end{center}
\caption{Scenario 1b. We consider the utility value $u_c=0.6$ with different costs for the external intervention $\kappa_e$, i.e. $0$, $50$ and $100$ euros per individual (or house) treated are considered.}
	\label{scenario1b}
\end{figure}
\begin{figure}[hbtp]
\begin{center}
	\includegraphics[width=0.9\textwidth]{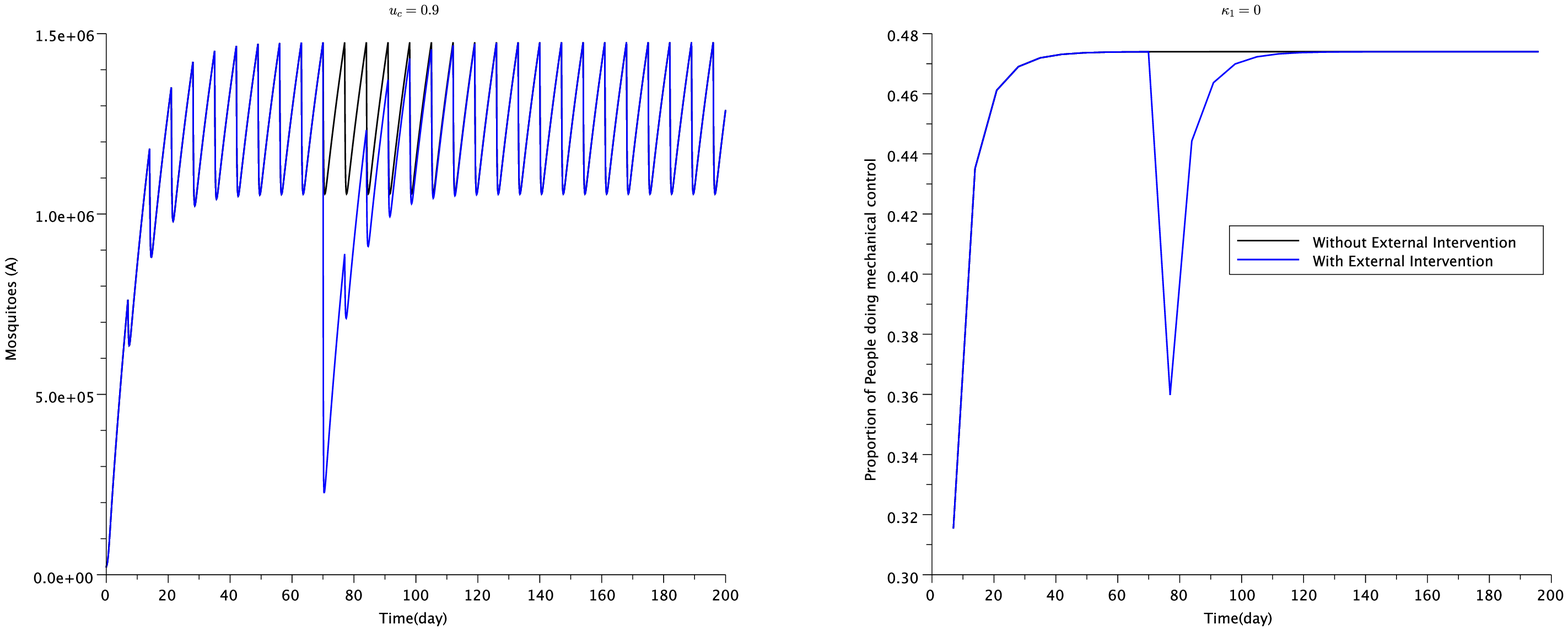} 
		\includegraphics[width=0.9\textwidth]{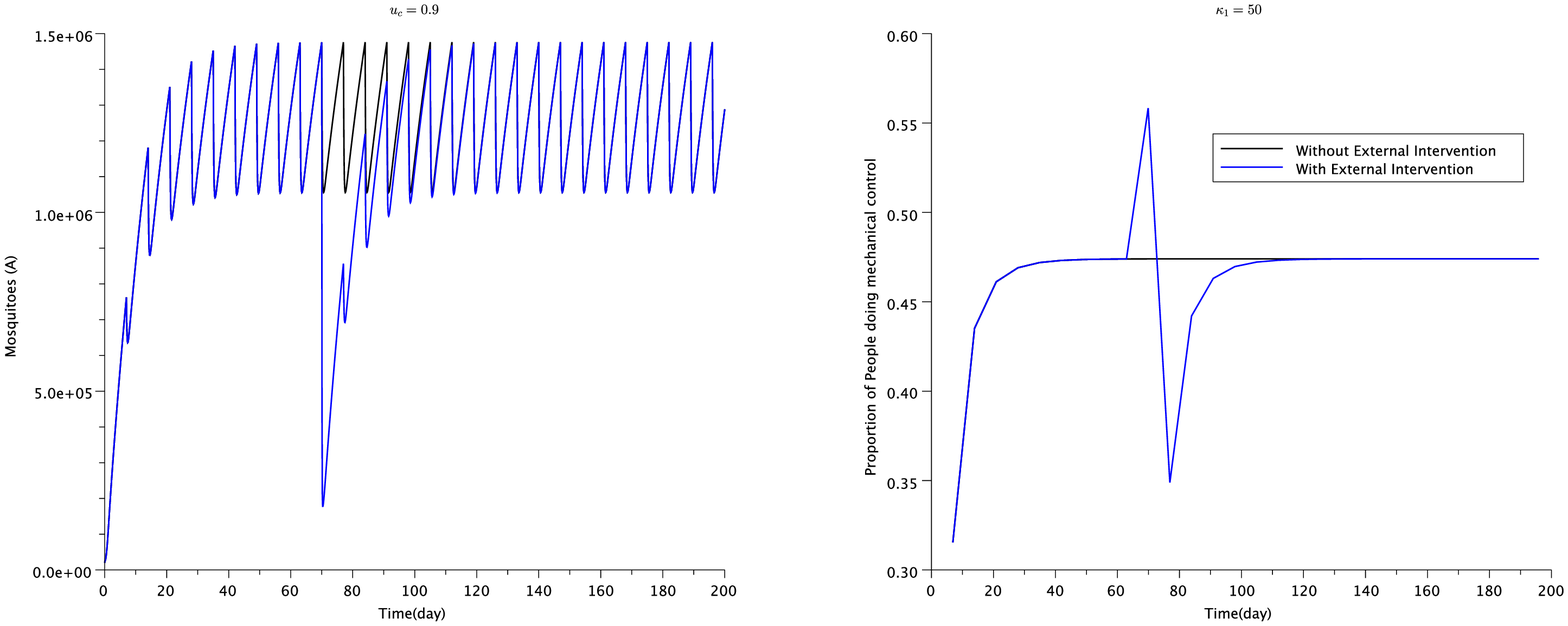} 
		\includegraphics[width=0.9\textwidth]{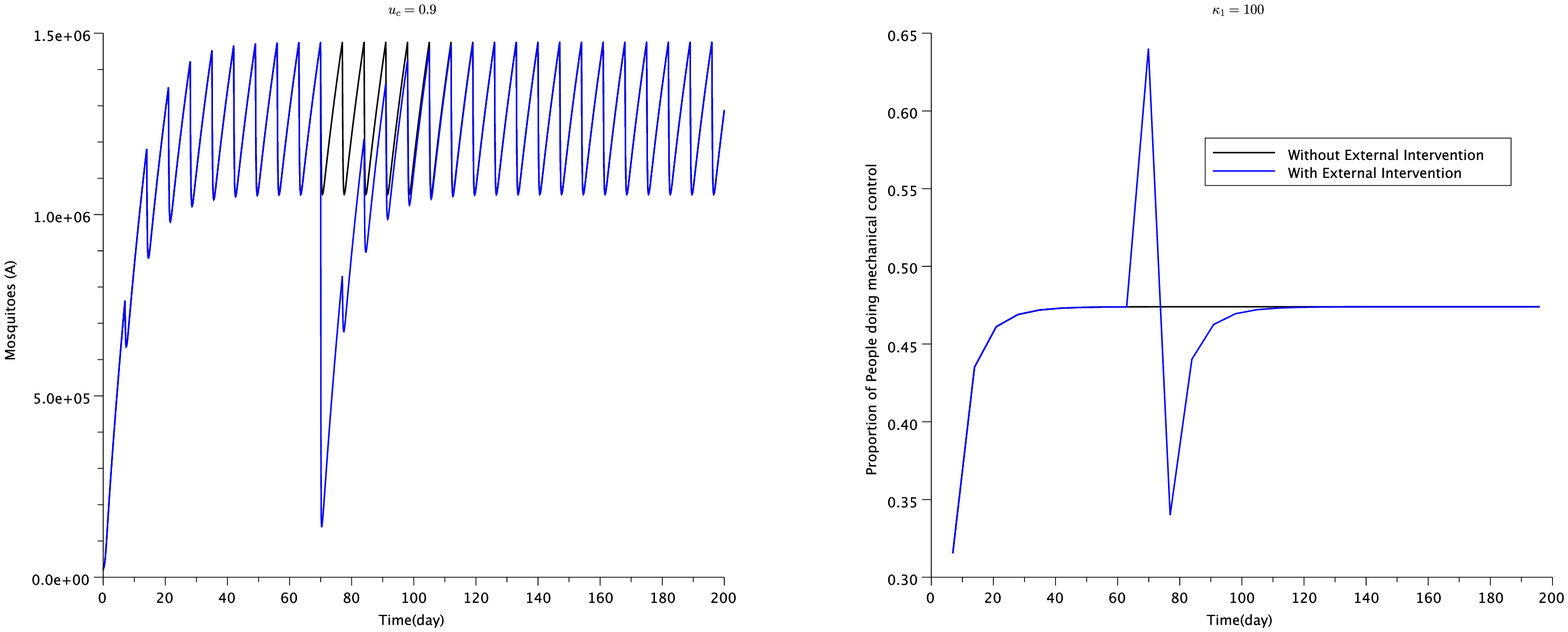} 
	
\end{center}
\caption{Scenario 1c. We consider the utility value $u_c=0.9$ with different costs for the external intervention $\kappa_e$, i.e. $0$, $50$ and $100$ euros per individual (or house) treated are considered.}
	\label{scenario1c}
\end{figure}
Whatever the value taken by $u_c$, the cost of the local intervention may impact $H$: the larger $\kappa_e$, the larger $H$. Altogether, the external intervention has an impact that last only for a couple of weeks, before the system reaches again the periodic equilibrium. 
\item 
In scenario 2 (Figure \ref{scenario2a}, we consider $u_c=50$ and $\kappa_e=60$. We assume a psychological effect of the intervention (such as a placebo effect for a clinical treatment), that is, the level of tolerance to mosquito bites is higher after the external intervention due to a psychological effect of the intervention. In other words, we shift the bites tolerance thresholds $k$ from $3$ to $6$, $9$ and $12$ (see remark \ref{remarkpiqure}) after the external intervention. Due to this effect, less and less people perform mechanical control themselves (which explains lower values for $H$ with external intervention than without). This implies automatically that the mosquito population shifts to a new periodic equilibrium, larger than before the external intervention, with oscillations of smaller amplitude.

\begin{figure}[hbtp]
\begin{center}
		\includegraphics[width=0.9\textwidth]{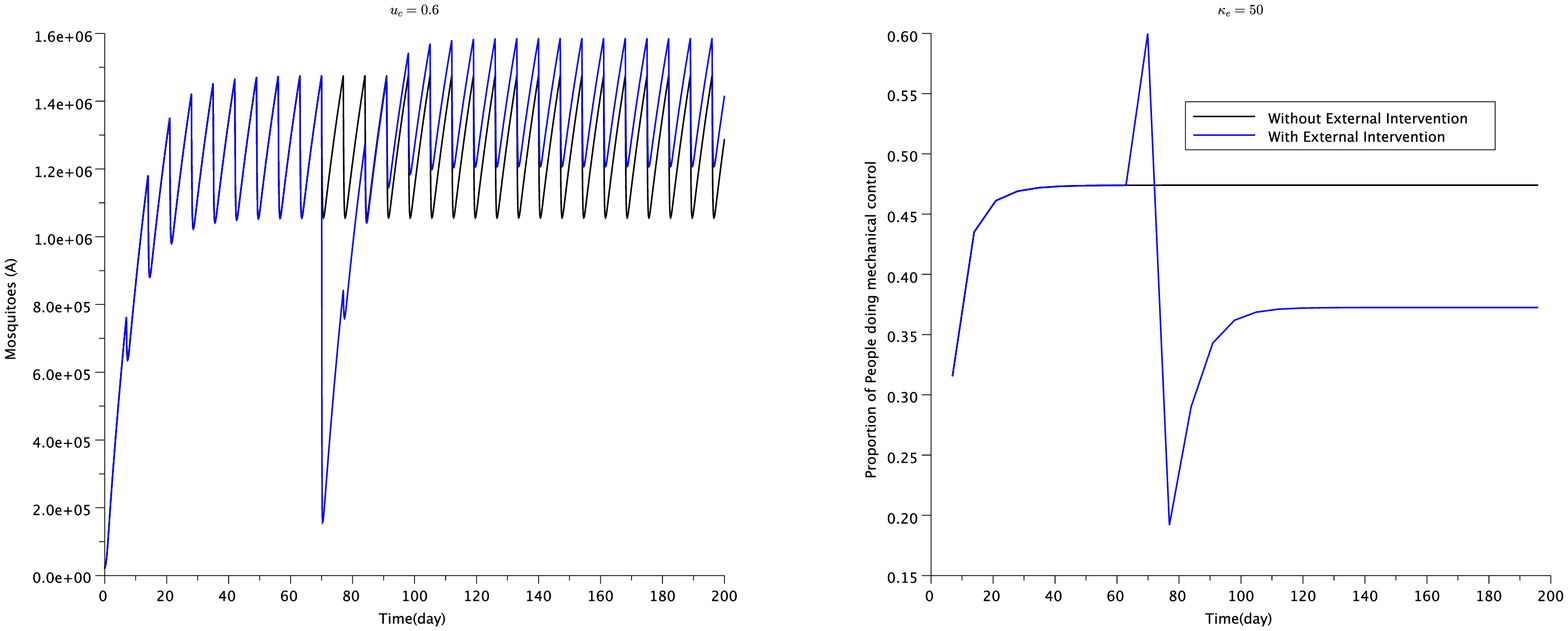} 
			\includegraphics[width=0.9\textwidth]{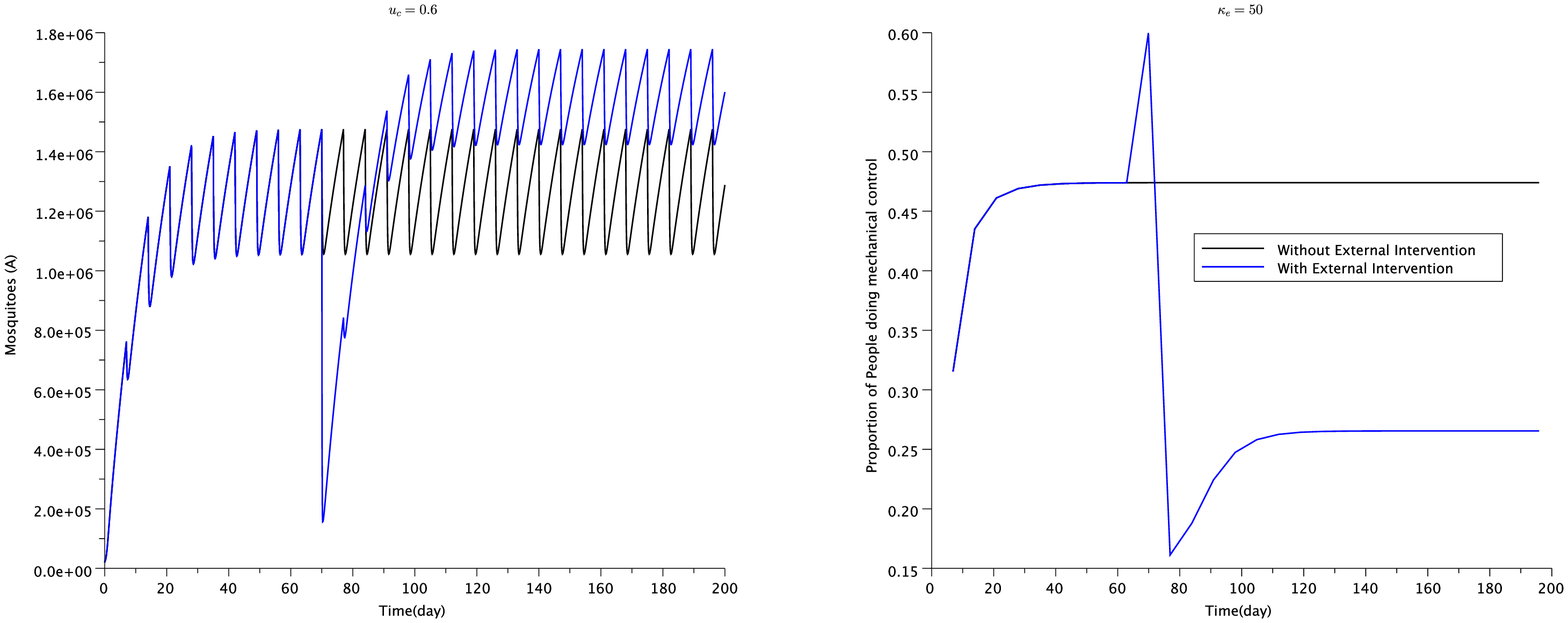} 
			\includegraphics[width=0.9\textwidth]{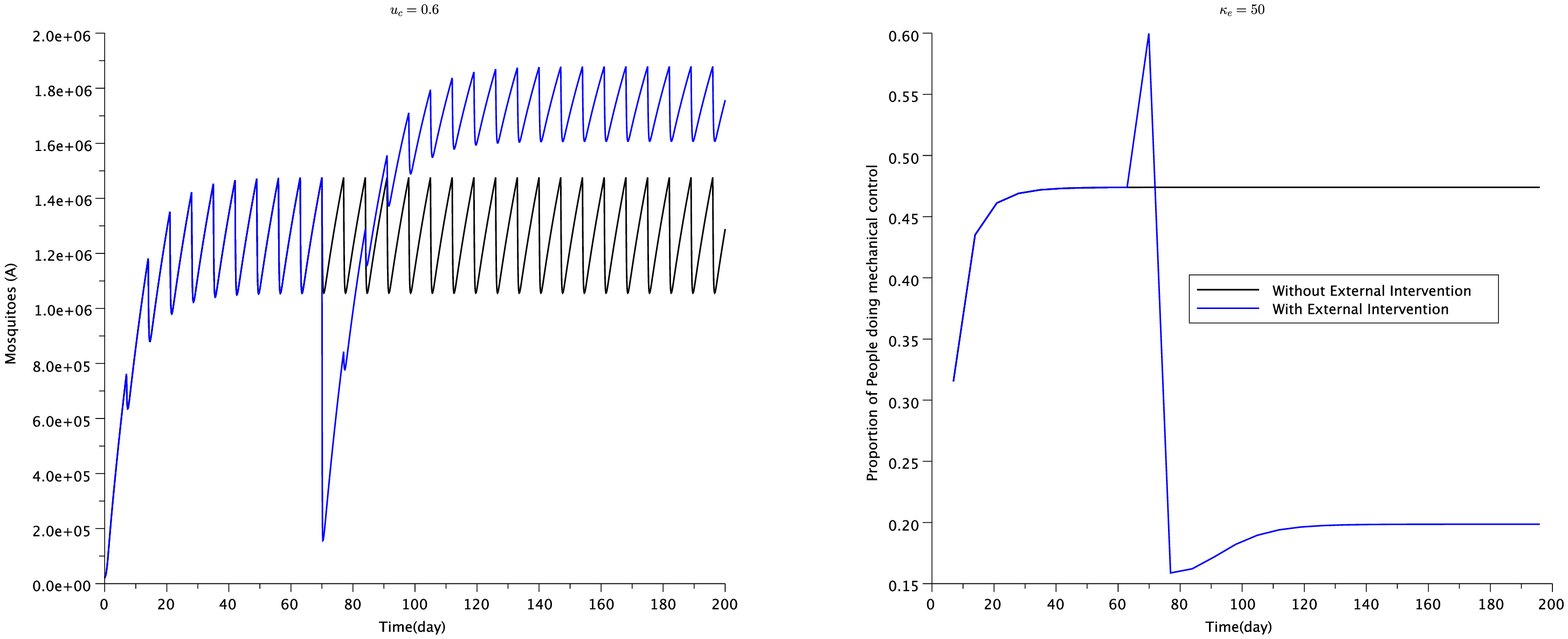}  
	\end{center}
		\caption{ Scenario 2. Different bites tolerance thresholds $k$ ($6$, $9$ and $12$)
		with an utility value $u_c=0.6$, and a cost for the external intervention $\kappa_e=50$ euros per individual (or house) treated.} 
			\label{scenario2a}
\end{figure}

\item In scenario 3 (Figure \ref{scenario3}, page \pageref{scenario3}), we consider $u_c=0.6$ and $\kappa_e=50$. We assume that people increase the larval capacity growth rate ($r_{K}=0.06$ after intervention) such that the periodic equilibrium is larger than without external intervention. For instance people may clear themselves of the responsibility to clean their larval gites because of the action made by the public agency. However, since the tolerance against mosquito bites does not change, $H_e$ is larger than $H$.
\begin{figure}[hbtp]
\centering
	\includegraphics[width=0.9\textwidth]{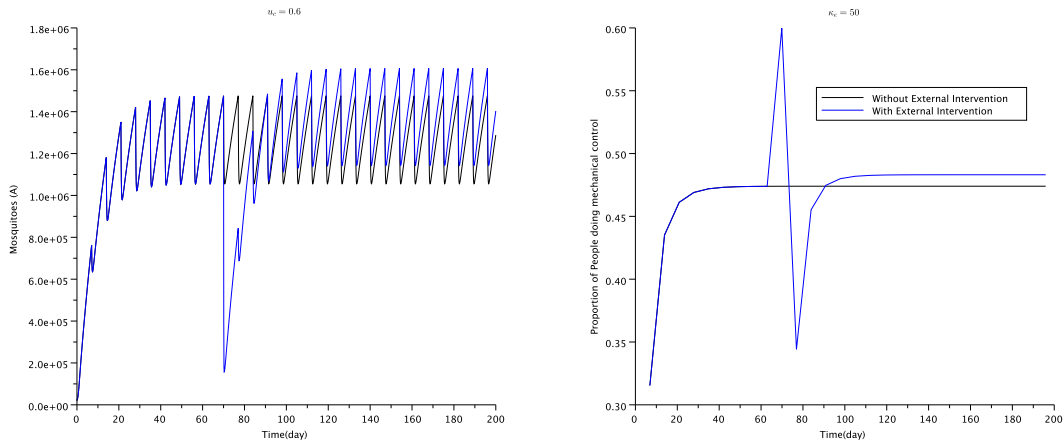} 
	\caption{Scenario 3.  ($u_c=0.6$ and $\kappa_e=50$) - Evolution of the mosquito population and proportion of people that will do mechanical control (with and without external intervention at day 70; local people make mechanical control every 7 days)}
		\label{scenario3}
\end{figure}

\item Scenario 4 (Figure \ref{scenario4}, page \pageref{scenario4}) is a mixed of the scenarii 2 and 3: we assume that the tolerance $k$ increase from $3$ to $6$, and $r_K$ increase from $0.05$ to $0.06$. Then, as expected, the mosquito population larger than before the external intervention. However we have $H_e <H$.

\begin{figure}[hbtp]
\centering
	\includegraphics[width=0.9\textwidth]{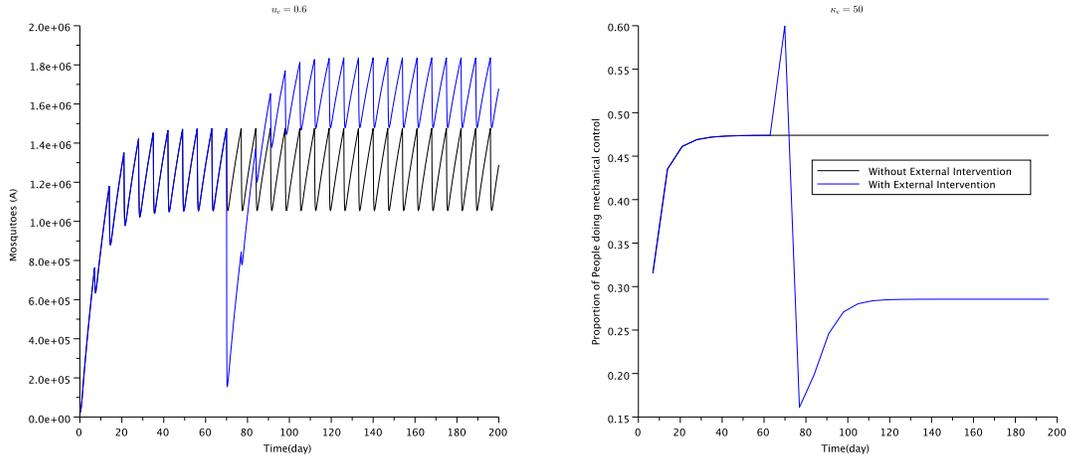} 
	\caption{Scenario 4.  ($u_c=0.6$ and $\kappa_e=50$) - Evolution of the mosquito population and Proportion of people that will do mechanical control (with and without external intervention at day 70; local people make mechanical control every 7 days)} when $r_K$ and $k$ increase.
		\label{scenario4}
	
\end{figure}

\end{itemize}
Also, it would be possible to consider $r_K$ and $K_{\max}$ (as well as the mosquito parameters) as periodic functions (according to dry and we seasons) or general function related to temperature and humidity data, the results obtained by numerical simulations in the previous scenarii still remain (or even can be worse!).
\section{Conclusion} \label{sec:5}
Our study aims to contribute to the understanding of public policy aspects of health. Though the public benefits of community-based healthcare strategies may be high, the immediate private returns of private behaviors could be low. Achieving widespread adoption of accurate preventive behaviors, then, might require costly sustained subsidies or taxation \cite{arrow_uncertainty_1963, gersovitz_tax/subsidy_2005}. In addition, public action may generate perverse incentives. Our model can highlight several scenarii that may be used to reveal household behaviors in response to a public intervention and give policy recommendations. Notably, an external public health intervention may increase stegomya and entomological indices in the presence of endogenous protective behaviors and, as a consequence,  public health interventions may have perverse effects. This may happen if the intervention induces a placebo effect, which is highly probable for many public health interventions involving human interactions, or if individuals clear themselves of the responsibility to clean their garden. The predictions of the model could be tested easily and a randomized experiment has been conducted recently for this purpose. Preliminary results in Réunion island seem to confirm this assertion. Altogether, human (individual and group) behaviors are important components that need to be taken into account in control Modelling. Our work bring some insights in this direction.

In \cite{dumont_mathematical_2012}, where Sterile Insect Technique (SIT) models for \textit{Aedes albopictus} was developed and studied, it was proved that coupling SIT with Mechanical Control perform (far) better results than without. According to the present work, it seems necessary to reevaluate these results, and to take into account the impact of human behaviors.

\vspace{1cm}

\noindent \textbf{Acknowledgments.}
YD has been partly supported by the SIT phase 2 pilot program in R\'eunion, jointly funded by the French Ministry of Health and the European Regional Development Fund (ERDF) under the 2014-2020 Operational Programme.

\bibliography{Biblio}


\end{document}